\documentclass{imammb}
\jno{dqnxxx}
\usepackage{natbib}
\usepackage{graphicx}
\usepackage{amsmath,amsthm,bm,mathrsfs}
\renewenvironment{proof}[1][Proof]{\noindent\textit{#1. } }{\hfill$\square$}

 \newtheoremstyle{theorem}{6pt}{6pt}{\rm}{}{\sffamily}{ }{ }{}
 \theoremstyle{theorem}
\newtheorem{theorem}{\sc Theorem}[section]

 \newtheoremstyle{algorithm}{6pt}{6pt}{\rm}{}{\sffamily}{ }{ }{}
 \theoremstyle{algorithm}

 \newtheoremstyle{lemma}{6pt}{6pt}{\rm}{}{\sffamily}{ }{ }{}
 \theoremstyle{lemma}
 \newtheorem{lemma}{\sc Lemma}[section]

\newtheoremstyle{case}{6pt}{6pt}{\rm}{}{\sffamily}{. }{ }{}
 \theoremstyle{case}

 \newtheoremstyle{statement}{6pt}{6pt}{\rm}{}{\sffamily}{ }{ }{}
\theoremstyle{statement}

 \newtheoremstyle{corollary}{6pt}{6pt}{\rm}{}{\sffamily}{ }{ }{}
 \theoremstyle{corollary}

  \newtheoremstyle{definition}{6pt}{6pt}{\rm}{}{\sffamily}{ }{ }{}
 \theoremstyle{definition}

\newtheoremstyle{example}{6pt}{6pt}{\rm}{}{\sffamily}{ }{ }{}
\theoremstyle{example}

\newtheoremstyle{remark}{6pt}{6pt}{\rm}{}{\sffamily}{ }{ }{}
\theoremstyle{remark}
\newtheorem{remark}{\sc Remark}[section]

\newtheoremstyle{approximation}{6pt}{6pt}{\rm}{}{\sffamily}{ }{ }{}
\theoremstyle{approximation}

\newtheoremstyle{scheme}{6pt}{6pt}{\rm}{}{\sffamily}{ }{ }{}
\theoremstyle{scheme}

\newtheoremstyle{Algorithm}{6pt}{6pt}{\rm}{}{\sffamily}{ }{ }{}
\theoremstyle{Algorithm}

\newtheoremstyle{Assumption}{6pt}{6pt}{\rm}{}{\sffamily}{ }{ }{}
\theoremstyle{Assumption}

\newtheoremstyle{proposition}{6pt}{6pt}{\rm}{}{\sffamily}{ }{ }{}
\theoremstyle{proposition}

\newtheoremstyle{hypo}{6pt}{6pt}{\rm}{}{\sffamily}{ }{ }{}
 \theoremstyle{hypo}

  \newtheoremstyle{Step}{6pt}{6pt}{\rm}{}{}{ }{ }{}
 \theoremstyle{Step}

\numberwithin{equation}{section}

\usepackage{tabularx}
\usepackage{comment}
\begin{document}
	\title{A deterministic model for non-monotone relationship between translation of upstream and downstream open reading frames}
	\author{ {\sc 	D. E. Andreev}\\[2pt]
		Moscow State Unversity, Moscow, Russia;\\
		Shemyakin-Ovchinnikov Institute of Bioorganic Chemistry, RAS, Moscow, Russia\\[6pt]
		{\sc 	P. V. Baranov}\\[2pt]
		 University College Cork, Cork, Ireland;\\ Shemyakin-Ovchinnikov Institute of Bioorganic Chemistry, RAS, Moscow, Russia\\[6pt]
		 {\sc 	A. Milogorodskii$^*$ and D. Rachinskii}\\[2pt]
		 University of Texas at Dallas, Richardson, TX, USA\\[6pt]
		 {\rm $^*$ Corresponding author. Email: axm170052@utdallas.edu}\\
		{\rm [Received on 28 December 2020]}\vspace*{6pt}}
	\pagestyle{headings}
	\markboth{D. E. Andreev, P. V. Baranov, A. Milogorodskii and D. Rachinskii}{\rm A deterministic model for relationship between translation of upstream and downstream ORF}
	\maketitle
	
\begin{abstract}
	{TASEP modeling was shown to offer a parsimonious explanation for the experimentally confirmed ability of a single upstream Open Reading Frames (uORFs) to upregulate downstream translation during the integrated stress response. As revealed by numerical simulations, the model predicts that reducing the density of scanning ribosomes upstream of certain uORFs increases the flow of ribosomes downstream. To gain a better insight into the mechanism which ensures the non-monotone relation between the upstream and downstream flows, in this work, we propose a phenomenological deterministic model approximating the TASEP model of the translation process. We establish the existence of a stationary solution featuring the decreasing density along the uORF for the deterministic model. Further, we find an explicit non-monotone relation between the upstream ribosome density and the downstream flow for the stationary solution in the limit of increasing uORF length and increasingly leaky initiation. The stationary distribution of the TASEP model, the stationary solution of the deterministic model and the explicit limit are compared numerically.}
	{mRNA translation regulation, stress resistance, ordinary differential equation model, asymptotics of stationary solution, modified TASEP model.}
\end{abstract}
\section{Introduction}
In eukaryotes, canonical translation initiation begins with the recognition of the m7G cap structure found at the 5’ end of mRNAs. This is followed by forming the preinitiation complex (PIC), which starts to ``scan'' the mRNA, unwinding mRNA secondary structures and probing mRNA for potential sites of translation initiation. After the initiation codon is recognized, the chain of events leads to large ribosome subunit joining and initiation of polypeptide synthesis (for more details see recent reviews \citep{1,2,3,4,5} on the mechanism of translation initiation in eukaryotes and its regulation). Once the ribosomes initiate translation, they continue moving along mRNA elongating the peptide chain (elongating ribosomes).

Not all translation initiation events lead to the synthesis of annotated functional proteins. Many codons recognized as starts of translation occur upstream of the annotated coding open reading frame (acORF), which encodes functional proteins in eukaryotic organisms \citep{6,7,8,9,10,11,12}. Abundant translation initiation upstream of acORFs has been confirmed by several ribosome profiling experiments \citep{19,20,21}. Such translation of an upstream open reading frame (uORF) can regulate eukaryotic gene expression. In particular, ribosome profiling revealed that translation of uORFs is altered in response to changes in physiological conditions. 

A number of stress conditions lead to a global increase in translation of uORFs on mRNA leaders \citep{22,23,24,25}, and 
sometimes reciprocal changes in acORF translation can be observed among individual mRNAs \citep{22}. One of the stress conditions where uORF-mediated translation control seems to be particularly important is in the integrated stress response (ISR) \citep{26,27,28}. 
During ISR the global translation is often down-regulated. This is achieved via molecular events that limit an important initiation factor required for PIC formation, leading to the reduced level of scanning ribosomes at mRNA. However, translation of a small number of mRNAs is resistant to, or upregulated, during ISR, and the most stress resistant mRNAs possess translated uORFs \citep{26,30}. 

The classical mechanism of uORF mediated stress resistance, known as delayed reinitiation, involves several uORFs, and at least two are absolutely essential for stress resistance \citep{32,33,34}. The archetypical example of this mechanism occurring in yeast has been characterized in details (reviewed in \citep{31}). On the other hand, examples have been found where only a single uORF is sufficient for imbuing an mRNA with translational stress resistance \citep{26, 35,36,37,38}, suggesting that there are other mechanisms of upregulating gene expression during stress. For instance, the start codon of the uORF is in a suboptimal Kozak context and allows for leaky scanning \citep{35,39} when only part of the ribosomes are initiated at the uORF start codon, while the others bypass it and begin translation at the start codon of the acORF. It has been hypothesized that the stringency of start codon recognition is increased during particular stress conditions, which allows for more leaky scanning, effectively increasing the flow of scanning ribosomes through the uORF and upregulating translation of the acORF \citep{35,39}.

It is also possible that elongating ribosomes translating the uORF obstruct progression of scanning ribosomes downstream, and this obstruction is relieved during stress due to reduced number of elongating ribosomes (albeit the stringency of the start codon recognition remains unchanged). In this scenario,  the intensities of the flows of scanning ribosomes upstream and downstream the uORF
are negatively correlated due to the obstruction of scanning ribosomes, which may potentially explain how a single uORF mediates stress resistance. However, it is {\em a priori} unclear whether such a mechanism is plausible without additional factors involved because stress reduces the number of both elongating and scanning ribosomes creating counteracting tendencies. While most stress resistant mRNAs possess uORFs, only very few uORF-containing mRNAs are stress resistant \citep{26}. This leads to the question what makes some uORFs to be the providers of stress resistance? In particular, can the interaction of scanning and elongating ribosomes trailing along the uORF explain the upregulation of downstream expression during stress?

The latter question was recently explored using a variant of the Totally Asymmetric Simple Exclusion Process (TASEP) model \citep{40}. Computer simulations of the model predicted that indeed a small subset of specific uORFs (constrained by length and leakiness of their initiation sites) are capable of upregulating translation downstream of a single uORF in response to the reduced availability of ribosomes upstream, which is observed under ISR. 

TASEP is a stochastic dynamical system of unidirectional particle movement through a one dimensional lattice, where each site can be occupied by no more than one particle, and the probability of particle transition from one site to another is predefined. Beginning with \citep{macdonald}, it has provided a basis for most mechanistic studies of mRNA translation dynamics and has been widely used for modeling a variety of dynamical systems including non-equilibrium transport phenomena, biological systems and road traffic \citep{41,42,43,44,45,46,465,47}. In the classical setting, the number of particles is conserved. In this case, important mathematical tools for the analysis of the stationary distribution 
of the TASEP model with the associated phase transitions are provided by the so-called hydrodynamic limit, which takes the form of a parabolic differential equation (for the density of the particles) expressing the conservation law \citep{derrida, blythe}. The hydrodynamic limit has been recently extended to a generalized version of the TASEP model (the $\ell$-TASEP with particles occupying multiple sites on the lattice and inhomogeneous jump rates), which was applied to analyze mRNA translation efficiency \citep{sch, song}. 

The variant of the $\ell$-TASEP model developed in \citep{40} was based on the assumption that collisions of elongating ribosomes with scanning ribosomes can result in the dissociation of the latter from the mRNA. Hence, the conservation is violated.
In this work, we propose a deterministic model of the uORF-mediated stress resistance based on a similar set of assumptions and additional simplifications. This system is less flexible than the probabilistic non-conservative $\ell$-TASEP model, which for example can  account for the fact that ribosomes occupy multiple sites on mRNA \citep{40}. On the other hand, the added value of the deterministic model is that it lends itself to rigorous analysis and admits an explicit solution in limit cases, providing further insight into a plausible mechanism of uORF-mediated stress resistance. Using this deterministic model, we confirm that leaky initiation combined with removal of scanning ribosomes due to collisions with elongating ribosomes can explain the non-monotone relationship between the flows of scanning ribosomes upstream and downstream of the uORF.

The deterministic model relates the rate of change of the average densities of scanning and elongating ribosomes at each mRNA site, $\rho_n^s$ and $\rho_n^e$, respectively, to the densities at neighboring sites. As described in Section \ref{2}, it is assumed that the flow of scanning ribosomes from site $n$ to the next site is proportional to the product of the density $\rho_n^s$  and the probability $1-\rho_{n+1}^s-\rho_{n+1}^e$ that the next site is unoccupied. The exit flow of scanning ribosomes from each site $n$ is assumed proportional to the product of the density $\rho_n^s$ and the density $\rho_{n-1}^e$ of elongating ribosomes at the previous site.

The paper is organized as follows. The next section presents the stochastic TASEP model and the deterministic model of uORF translation. Section \ref{mainn} contains main results. We show that the deterministic model has a stationary solution with certain properties and find this solution in the limit of increasing mRNA length. This allows us to infer conditions that ensure the non-monotone dependence of the downstream flow on the upstream flow of scanning ribosomes. In Section \ref{numerics}, we compare numerically the predictions of the TASEP and deterministic models with the explicit limit solution. Section \ref{proofs} contains the proofs.

\section{Models}\label{2}
\subsection{Modified TASEP model for propagation of ribosomes along uORF}\label{model}
We represent scanning and elongating ribosomes as two separate types of particles with different dynamic properties with the possibility of transformation of one into the other at specific sites. The uORF is modeled as a discrete array of positions numbered from $0$ to $N$. Each position represents one codon. Position $N_1$ corresponds to the start codon, position $N_1+N_2$ corresponds to the stop codon. Particles of the first kind (representing scanning ribosomes) proceed from position $0$ to position $N_1$, where part of them is converted to particles of the second kind (representing elongating ribosomes), see Figure \ref{tasep}. Then, particles of both kinds proceed to the position $N_1+N_2$, where particles of the second kind are eliminated. Particles of the first kind continue to the position $N_*-1$ and are eliminated there. Hence, particles of the second kind can be encountered between the positions $N_1$ and $N_1+N_2$ only.
\begin{figure}[hbt]
	\centering
	\includegraphics[scale=0.55]{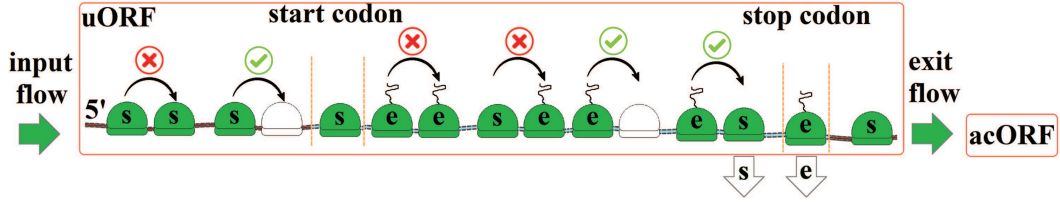}
	\caption{Block scheme of the TASEP model of the uORF translation.}
	\label{tasep}
\end{figure}

As the first model, we consider a modification of the TASEP stochastic process. The state of the model at time $k\in \mathbb{N}_0$ is a vector $s^k=(s_0^k,\ldots,s_{N_*-1}^k)$ with $s^k_i\in \{0,1,2\}$, where $s_i^k=0$ corresponds to the vacant (unoccupied) codon, $s_i^k=1$ corresponds to the codon occupied by a particle of the first kind and $s_i^k=2$ corresponds to the codon occupied by a particle of the second kind at the position $i$. At every time step, a position $n=n_k$ is selected at random with equal probability $1/N_*$. If either $s_{n_k}^k\ne 0$ or $n_k=0$, then the components $s_{n_k}^k$ and $s_{n_k+1}^k$ of the state are changed according to the rules listed in Table \ref{T1} (in the case $n_k=N_*-1$, only $s^k_{N_*-1}$ is changed). The other components remain unchanged, i.e. $s^{k+1}_m=s^k_m$ for $m\ne n_k, n_k+1$. 

In the cases, which are not listed in Table \ref{T1}, the state remains unchanged. In particular, $s^{k+1}=s^k$ if either $s^k_{n_k}=0$ for $n_k\ne 0$ or $(s_{n_k}^k, s_{n_k+1}^k)\in \{(1,1), (1,2), (2,2)\}$.
\vspace*{-7pt}
\begin{table}[!h]
	\caption{Probabilities of transitions $(s^k_{n_k},s^k_{n_k+1})\to (s^{k+1}_{n_k},s^{k+1}_{n_k+1})$}
	\label{T1}
	\begin{tabularx}{\textwidth}{X l}
		\hline
		$n_k\ne 0, N_1-1, N_1+N_2-1, N_*-1:$& \\
		&$(1,0) \to (0,1)$ with probability $1$;\\
		&$(2,0) \to (0,2)$ with probability $1$;\\
		&$(2,1) \to (0,2)$ with probability $1$;\\
		\hline
		$n_k=0:$&  \\ 
		&$(0,0) \to (1,0)$ with probability $\rho^0\in (0,1]$; \\ 
		&$(0,1) \to (0,1)$ with probability $1-\rho^0$; \\ 
		&$(1,0) \to (1,1)$ with probability $\rho^0$; \\ 
		&$(1,0) \to (0,1)$ with probability $1-\rho^0$; \\ 
		\hline
		$n_k=N_1-1:$&  \\ 
		&$(1,0) \to (0,2)$ with probability $c\in (0,1)$;  \\ 
		&$(1,0) \to (0,1)$ with probability $1-c$;  \\ 
		\hline
		$n_k=N_1+N_2-1:$&\\
		&$(1,0)\to (0,1)$ with probability $1$;\\
		&$(2,0) \to (0,0)$ with probability $1$;\\
		&$(2,1)\to (0,1)$ with probability $1$;\\
		\hline
		$n_k=N_*-1:$&\\
		&$s_{N_*-1}^k=1 \to s_{N_*-1}^{k+1}=0$ with probability $1$.\\
		\hline
	\end{tabularx}
	\vspace*{-4pt}
\end{table}

These rules imply that a particle of the first kind, if selected at a particular time step, moves from its position $n_k$ to the next position $n_k+1$ provided that the latter is vacant. A particle of the second kind moves to the next position if this position is either vacant or occupied by a particle of the first kind. In the latter case, the particle of the second kind replaces the particle of the first kind eliminating it from the array. Special rules apply for $n_k=0, N_1-1, N_1+N_2-1, N_*-1$. The initial position $n=0$, if vacant, is refilled by a particle of the first kind with the probability $\rho^0$. A particle of the first kind is converted into a particle of the second kind with the probability $c$ when moved from the position $N_1-1$ to $N_1$. Particles of the second kind are eliminated at the position $N_1+N_2-1$, and particles of the second kind are removed at the position $N_*-1$. The model is initiated with the state $s^0=(1,0,\ldots,0)$.

Results of numerical simulation of this model are presented in Section \ref{numerics}.

\subsection{Deterministic model}
Now, we consider a phenomenological deterministic model, which approximates the above stochastic process.

\begin{figure*}[hbt]
	\centering    	
	\includegraphics[scale=1.1]{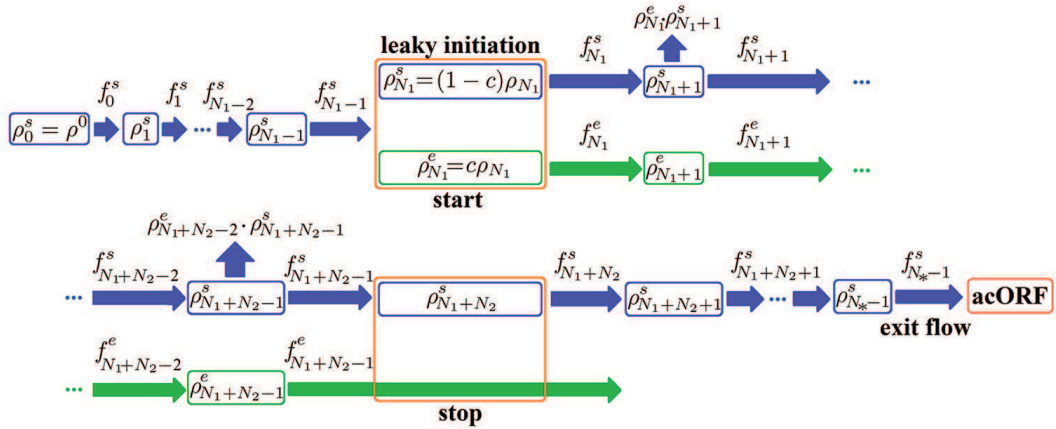}
	\caption{Block scheme of the deterministic model of the uORF translation.}
\end{figure*}

By $\rho_n^s$ and $\rho_n^e$ we denote the densities of particles of the first and second kind, respectively, at position $n$. These densities satisfy $0\le \rho_n^s,\quad \rho_n^e \le \rho_n^s+\rho_n^e\le 1$. The boundary conditions for particles of the first kind are
\begin{equation}\label{bvr}
\rho_0^s=\rho^0, \qquad \rho_{N_*}^s=0.
\end{equation}
Since there are no particles of the second kind outside the segment between the start and stop codons, we require
\begin{equation}\label{bvR}
\rho_0^e=\rho_1^e=\cdots=\rho_{N_1-1}^e=\rho_{N_1+N_2}^e=\rho_{N_1+N_2+1}^e=\cdots=
\rho_{N_*}^e=0.
\end{equation}

It is postulated that the one-directional flow of particles of the first kind from position $n$ to position $n+1$ equals 
\[
f_n^s=\rho_n^s(1-\rho_{n+1}^s-\rho_{n+1}^e), \qquad n=0,1,\ldots,N_*-1;
\]
the flow of particles of the second kind from position $n$ to position $n+1$ is
\[
f_n^e=\rho_n^e(1-\rho_{n+1}^e), \qquad n=0,1,\ldots,N_*-1.
\]
Using these formulas for the flows, the balance equation for particles of the first kind at any position $n\ne N_1$ reads
\begin{equation}
\label{eqr}
\frac{1}{v}\dot{\rho}_n^s=
f_{n-1}^s-f_n^s-\rho_{n}^s\rho_{n-1}^e, \qquad n=1,\ldots,N_*-1; \ \ n\ne N_1,
\end{equation}
where $v$ is the velocity of the particles, and the last term, $-\rho_{n}^s\rho_{n-1}^e$, accounts for collisions of particles of the second kind with particles of the first kind. It is assumed that such a collision results in the elimination of the particle of the first kind. The balance equations for particles of the second kind for $n$ between the start and the stop codon are
\begin{equation}\label{eqR}
\tfrac1v \dot {\rho}_n^e = f_{n-1}^e-f_n^e, \qquad n=N_1+1,\ldots,N_1+N_2-1.
\end{equation}

At $n=N_1$ (the start codon), the balance equation is slightly different:
\begin{equation}\label{eqs}
\dot \rho_{N_1}=f_{N_1-1}^s-f_{N_1}^s-f_{N_1}^e,
\end{equation}
where $\rho_{N_1}=\rho_{N_1}^s+\rho_{N_1}^e$ is the total density of particles at the start codon.

Balance equations \eqref{eqr}\,--\,\eqref{eqs} and boundary conditions \eqref{bvr}\,--\,\eqref{bvR} are coupled with the matching conditions at the start codon:
\begin{equation}\label{mc}
\rho_{N_1}^s =(1-c)\rho_{N_1},\qquad \rho_{N_1}^e = c \rho_{N_1},
\end{equation}
where the parameter $c\in (0,1)$ determines the fraction of particles of the first kind converted to particles of the second kind at the start codon.

\section{Main results}\label{mainn}

\subsection{Existence of a stationary solution}
We consider the stationary solution of problem \eqref{bvr}\,--\,\eqref{mc}. For this solution the flow of particles of the first kind is constant on the segments $0\le n\le N_1-1$ and $N_1+N_2\le n\le N_*-1$ and the flow of particles of the second kind is constant between the start and stop codons:
\begin{equation}\label{fj12}
f_0^s=\cdots=f_{N_1-1}^s=j^{(1)}; \qquad f_{N_1}^e=\cdots= f_{N_1+N_2-1}^e=J^{(2)};
\end{equation}
\begin{equation}\label{fj3}
f_{N_1+N_2}^s=\cdots=f_{N-1}^s=j^{(3)}.
\end{equation}
It will be convenient to slightly change the notation and denote
\[
\rho^{(1)}_n = \rho_n^s, \qquad n=0,\ldots,N_1-1; \qquad \rho^{(1)}_{N_1}=\rho_{N_1};
\]
\[
\rho^{(2)}_m=\rho_{N_1+m}^s, \quad R_m=\rho_{N_1+m}^e, \qquad m=0,\ldots,N_2+N_3,
\]
where $N_3=N_*-N_1-N_2$, as well as to introduce the new variables
\[
j_m^{(2)}=\rho_m^{(2)}(1-\rho_m^{(2)}), \qquad m=0,\ldots,N_2+N_3-1.
\]
That is, $\rho_n^{(1)}$ are densities of particles of the first kind upstream of the start codon, $\rho_m^{(2)}$ are densities of particles of the first kind downstream of the start codon and $R_m$ are densities of particles of the second kind downstream of the start codon. Note that $j_m^{(2)}$ is not equal to the flow $f_{N_1+m}^s$.

With this notation, the stationary solution satisfies the system of equations
\begin{equation}\label{problem}
\rho_{n+1}^{(1)}=1-\frac{j^{(1)}}{\rho_n^{(1)}}, \qquad n=0,\ldots, N_1-1;
\end{equation}
\begin{equation}\label{problem'}
R_{m+1}=
\begin{cases}
1-\frac{J^{(2)}}{R_m}, &\qquad m=0,\ldots, N_2-1;\\
R_m, &\qquad m=N_2,\ldots, N_2+N_3-1;
\end{cases}
\end{equation}
\begin{equation}\label{problem''}
\begin{cases}
\rho_{m+1}^{(2)}=1-\frac{j_m^{(2)}}{\rho_m^{(2)}}, &\quad m=0,\ldots,N_2+N_3-1;\\
j_{m+1}^{(2)}=j_m^{(2)}-R_{m+1}\rho_m^{(2)}-(R_{m}-R_{m+2})\rho^{(2)}_{m+1},
&\quad m=0,\ldots,N_2+N_3-2,
\end{cases}
\end{equation}
coupled with the boundary and matching conditions
\begin{equation}\label{matching}
\begin{aligned}[c]
\rho_0^{(1)}=\rho^0;\quad \ R_0={c}\rho_{N_1}^{(1)};\quad \ 
\rho_0^{(2)}=\left(1-{c}\right)\rho_{N_1}^{(1)};
\\
j^{(1)}=j_0^{(2)}+J^{(2)}-\rho_0^{(2)}R_1; \quad \ R_{N_2}=0;\quad \ 
\rho_{N_2+N_3}^{(2)}=0.
\end{aligned}
\end{equation}

A solution of problem \eqref{problem}\,--\,\eqref{matching} is called {\em positive} if all its unknown components are positive:

\begin{equation}\label{solution}
\rho_0^{(1)},\ldots, \rho_{N_1}^{(1)}, \rho_0^{(2)},\ldots, 
\rho_{N_2+N_3-1}^{(2)}, R_0,\ldots,R_{N_2-1}, j^{(1)}, j^{(2)}_0,\ldots, 
j^{(2)}_{N_2+N_3-1}, J^{(2)}>0.
\end{equation}
A solution is called {\em decreasing} if
\[
\rho^{(1)}_0 > \cdots > \rho^{(1)}_{N_1} > \rho^{(2)}_0 > \cdots > 
\rho^{(2)}_{N_2+N_3} =0; \qquad R_0 > \cdots > R_{N_2}=0.
\]

\begin{theorem}\label{t1}
	Given any $\rho^0\in (0,1/2]$, $c\in (0,1)$, problem 
	\eqref{problem}\,--\,\eqref{matching} has a positive decreasing solution.
\end{theorem}

This theorem is proved in the next section.

\subsection{Continuous limit of the stationary solution}

Denote by $W: [-e^{-1},\infty) \to [-1,\infty)$ the Lambert $W$-function, which is the inverse of the function $W^{-1}(w)=w e^w$.
Denote
\begin{equation}\label{varphi}
\varphi(\rho)=\rho(1-\rho),
\end{equation}
\begin{equation}\label{W'}
\rho_*(\rho^0,\tau)=-\frac12 W\left(-2\rho^0 e^ {-\rho^0\left(2+{c_0}\tau\right)}\right).
\end{equation}
Below we use the notation $x_n=O(N_2^\alpha)$ if $|x_n|\le C N_2^\alpha$, where $C$ is independent of $N_2$ for a given range of the values of $n$.

\begin{theorem}\label{t2}
	Suppose that $c=c_0/N_2$ with a constant $c_0>0$. Let $0<\rho^0< 1/2$.
	Then, the positive decreasing solution of problem \eqref{problem}\,--\,\eqref{matching}
	satisfies
	\begin{equation}\label{lim}
	\rho_n^{(1)}= \rho^0+O(N_2^{-1}), \qquad 0\le n\le N_1,
	\end{equation}
	\begin{equation}\label{lim'}
	\rho^{(2)}_m= \rho_*\left(\rho^0,\frac{m}{N_2}\right)+O(N_2^{-1}), \qquad 0\le m\le N_2 -\frac{\ln N_2}{\ln (1/\rho_0-1)},
	\end{equation}
	\begin{equation}\label{fj33}
	j^{(1)}=\varphi(\rho^0)+O(N_2^{-1}), \qquad 
	\end{equation}
	\begin{equation}\label{fj333}
	j^{(2)}_m = \varphi\left(\rho_*\left(\rho^0,\frac{m}{N_2}\right)\right)+O(N_2^{-1}\ln N_2), \qquad 0\le m\le N_2,
	\end{equation}
	\begin{equation}\label{fj3sq}
	j^{(3)}=\varphi\left(\rho_*\left(\rho^0,1\right)\right)+O(N_2^{-1}\ln N_2)
	\end{equation}
	as $N_2\to\infty$, where the functions $\varphi, \rho_*$ are defined by \eqref{varphi}, \eqref{W'}.
\end{theorem}

The proof is presented in the next section.

\begin{remark}\label{cor}
	It is easy to see that the limit exit flow $j^{(3)}_*(\rho^0)=\varphi(\rho_*(\rho^0,1))$ achieves its maximum over the interval $0\le \rho^0 \le 1/2$ at the point
	\begin{equation}\label{**}
	\rho^0=\frac{1}{2+c_0}.
	\end{equation}
	Thus, formula \eqref{fj3sq} of Theorem \ref{t2} ensures that the exit flow $j^{(3)}=f_{N_*-1}^s$
	of scanning ribosomes defined by the stationary solution of model \eqref{bvr}\,--\,\eqref{mc} reaches a maximum value when the initial density is close the value \eqref{**} if  $N_2$ is sufficiently large. In other words, the dependence of the exit flow on the initial density of scanning ribosomes is non-monotone.
\end{remark}

\begin{remark}\label{newremark2}
	{If $N_2$ is sufficiently large and $N_3=0$ (the uORF ends at the stop codon), then according to Theorem \ref{t2}, the density $\rho_m^{(2)}$ of scanning ribosomes drops from approximately the value $\rho_*(\rho^0,1)>0$ to zero over a relatively short segment of the uORF adjacent to its downstream end. In other words, the distribution of the density of scanning ribosomes over uORF has a boundary layer at the downstream end. The relative width of this boundary layer according to Theorem \ref{t2} is of the order of $N_2^{-1}\ln N_2$. }
	
	{The situation is similar for $N_3>0$. In this case, on the uORF segment downstream of the stop codon the density increments satisfy
		\begin{equation}
		0\le \rho_{m-1}^{(2)}-\rho_{m}^{(2)}\le \frac{\rho^0}{1-\rho^0} (\rho_m^{(2)}-\rho_{m+1}^{(2)}).
		\end{equation}
		Due to this relationship, for sufficiently large $N_3$, the density distribution exhibits one boundary layer, which is located at the downstream end of the uORF.
	}
\end{remark}

\begin{remark}\label{newremark3}
	The biologically relevant situation corresponds to relatively small values of $\rho^0$ (the low density regime). For completeness, let us discuss larger densities as well.
	
	The continuous limit of the standard TASEP model with the conserved flow undergoes a bifurcation as the density increases. The saturated congestion regime corresponds to all densities exceeding $0.5$ and is characterized by the same maximal flow of particles. This applies to the initial segment $0\le n\le N_1$ of the TASEP model described in Section \ref{model}. Therefore, the exit flow of this model has the same value for all the initial densities $\rho^0\ge 0.5$. Numerical examples are considered in the next section.
\end{remark}

\section{Numerical results}\label{numerics}

Figure \ref{inout} shows the dependence of the exit flow $j^{(3)}=j^{(3)}(\rho^0)$ of  scanning ribosomes at $n=N_*$
on the initial density $\rho^0$ at $n=0$. The function $j^{(3)}=j^{(3)}(\rho^0)$ is obtained from the stationary distribution of the TASEP model and the stationary solution of the approximating deterministic model \eqref{bvr}\,--\,\eqref{mc}.

\begin{figure}[!hbt]
	\centering
	\includegraphics[scale=1]{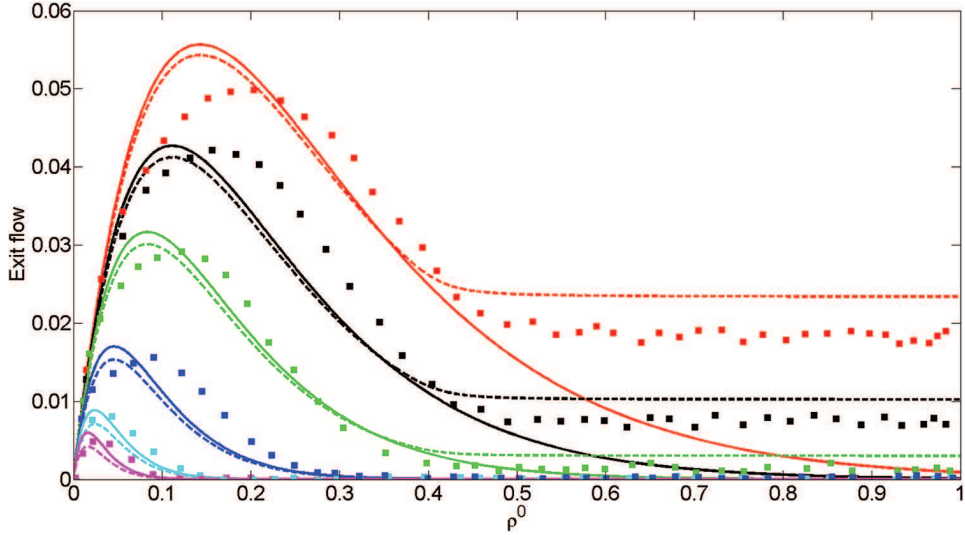}
	\caption{Exit flow of scanning ribosomes as a function of the density $\rho^0$ of ribosomes at $n=0$ for different values of the conversion rate, $c$, of scanning ribosomes to elongating ribosomes at the start codon. Dots correspond to the stationary distribution of the TASEP model, dashed curves correspond to the stationary solution of the deterministic model \eqref{bvr}\,--\,\eqref{mc}, and solid curves correspond to the explicit formula for the limit exit flow, $j_*^{(3)}(\rho^0)=\varphi(\rho_*(\rho^0,1))$, of Theorem \ref{t2} and Remark \ref{cor}. The red, black, green, blue, cyan and pink curves correspond to $c=0.025, 0.035, 0.05, 0.1, 0.2, 0.3$, respectively; $N_1=N_3=100$, $N_2=200$. }
	\label{inout}
\end{figure}

\begin{figure}[!hbt]
	\centering
	\includegraphics[scale=1]{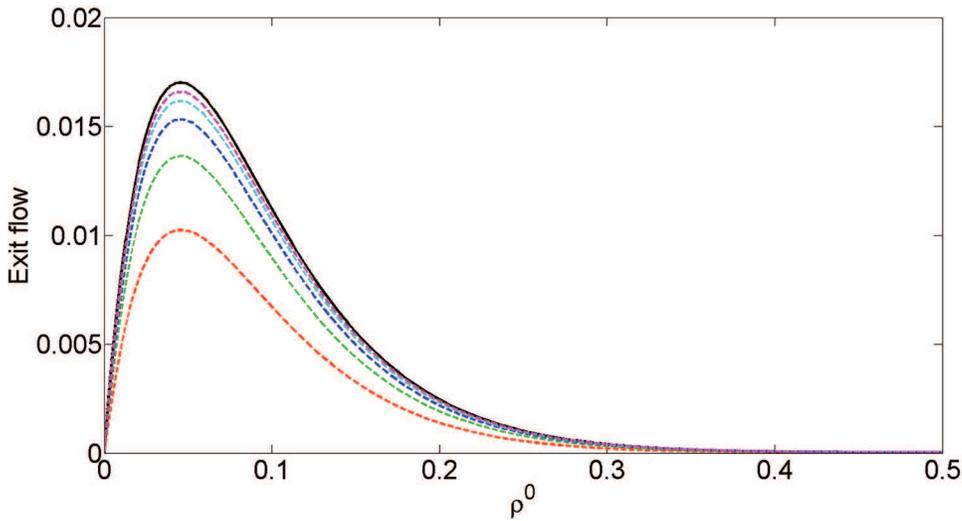}
	\caption{Exit flow of scanning ribosomes as a function of the density $\rho^0$ for different length of the uORF. The conversion rate of scanning ribosomes to elongating ribosomes scales with $N_2$ as $c=c_0/N_2$ where $c_0=20$ is fixed, see in Theorem \ref{t2}. The black solid curve corresponds to the explicit formula for the limit exit flow, $j_*^{(3)}(\rho^0) = \varphi(\rho_*(\rho^0,1))$, see Remark \ref{cor}. The dashed curves are obtained from the stationary solution of the deterministic model \eqref{bvr}\,--\,\eqref{mc}. The pink, cyan, blue, green and red curves correspond to $N_2 = 50$, $100$, $200$, $400$, $800$, respectively; $N_1=N_3=100$.}
	\label{statvsdeter}
\end{figure}

The exit flow peaks between $\rho^0=0$ and $\rho^0=0.5$. One can see that the peak predicted by the TASEP model is shifted to the right of the peak obtained from the deterministic model, and tends to be wider for the TASEP model. As expected, the output decreases with the increasing rate $c$ of conversion of scanning to elongating ribosomes at the start codon. 

In accordance with Theorem \ref{t2}, the dependence of the exit flow on the initial density, $j^{(3)}(\rho^0)$, for the stationary solution of the deterministic model \eqref{bvr}\,--\,\eqref{mc} is well approximated by the composition
$\varphi \circ \rho_*$ of the functions \eqref{varphi} and \eqref{W'} on the interval $0<\rho^0<1/2$, see Figures \ref{inout} and \ref{statvsdeter}.

According to Remark \ref{newremark3}, the stationary distribution of the TASEP model features (approximately) the same exit flow of scanning ribosomes for all the initial densities $1/2\le \rho^0<1$. The same is true for the deterministic model \eqref{bvr}\,--\,\eqref{mc}, see Figure \ref{inout}.
\begin{figure}[!hbt]
	\centering
	\includegraphics[scale=1]{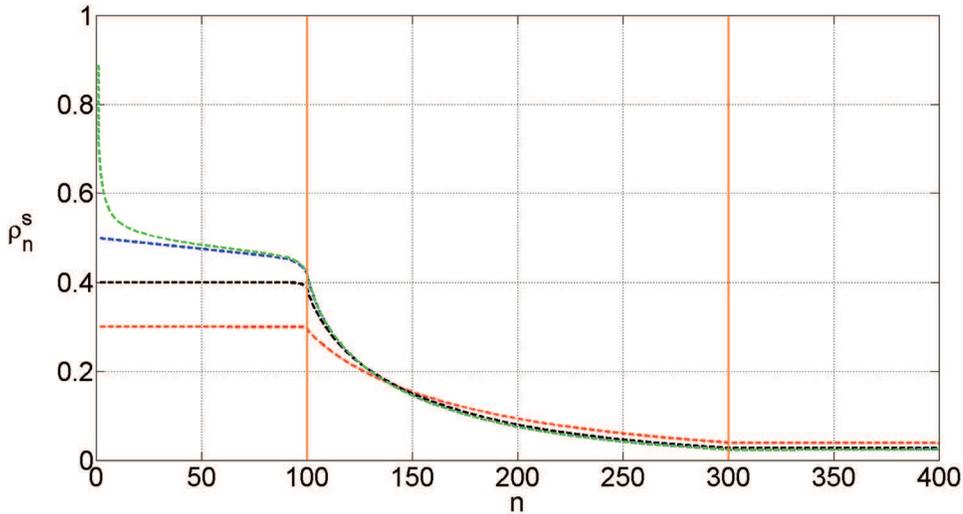}
	\caption{The dependence of the density $\rho_n^s$ on the position $n$ along the uORF for the stationary solution of the deterministic model \eqref{bvr}\,--\,\eqref{mc}. The red, black, blue and green curves correspond to $\rho^0 = 0.3$, $0.4$, $0.5$, $0.9$, respectively; $N_1=N_3=100$, $N_2=200$; $c=0.025$. Vertical lines indicate the position of the start and stop codons.  }
	\label{distonmrna}
\end{figure}

~~~~Figures \ref{distonmrna} and \ref{flowonmrna} present the stationary distribution of density and flow of scanning ribosomes, respectively, for the deterministic model \eqref{bvr}\,--\,\eqref{mc} for different values of the initial density and $N_2=200$. On the interval $N_1\le n\le N_1+N+2$, each density graph is practically indistinguishable from the limit density curve $\rho_n=\rho_*(\rho^0,n/N_2)$ with the corresponding $\rho^0$ (not shown). This agrees with Theorem \ref{t2}. Similarly, the flow graphs are very close to the limit curves $f_n=\varphi(\rho_*(\rho^0,n/N_2))$ corresponding to the limit $N_2\to\infty$. The flow graphs with $\rho^0\ge 1/2$ practically coincide with each other (the blue and green curves in Figure \ref{flowonmrna}). The flow reaches its maximum value $1/4$ on the initial segment $0\le n\le N_1$ for $\rho_0\ge 1/2$. The density graphs with $\rho^0\ge 1/2$ feature a boundary layer near $n=0$ and almost coincide with each other for larger values of $n$ (see the plots of the same color in Figure \ref{distonmrna}).

\begin{figure}[!hbt]
	\centering
	\includegraphics[scale=1]{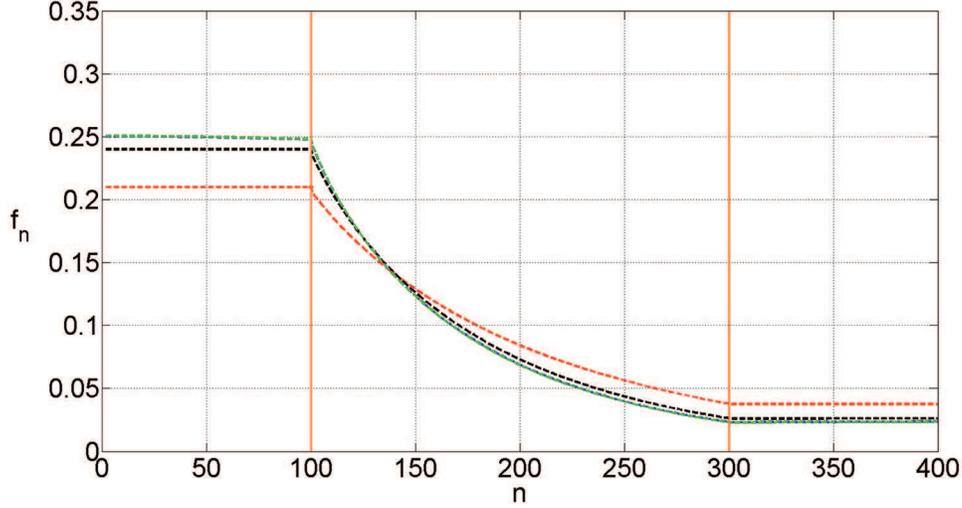}
	\caption{The dependence of the flow $f_n$ of scanning ribosomes on the position $n$ along the uORF for the stationary solution of the deterministic model \eqref{bvr}\,--\,\eqref{mc}. The color code and density are the same as in Figure \ref{distonmrna}.}
	\label{flowonmrna}
\end{figure}

\section{Proofs}\label{proofs}

\subsection{Auxiliary lemmas}
The proof of Theorem \ref{t1} uses a few auxiliary lemmas.

Consider iterations
\begin{equation}\label{J}
\rho_{n+1}=1-\frac{j}{\rho_n},\qquad \rho_0=\rho^0, \qquad 
n=0,1,2,\ldots
\end{equation}
Since the sequence \eqref{J} is defined by $j$ and $\rho^0$, we write $\rho_n=\rho_n(j,\rho^0)$. The following lemma summarizes simple properties of this sequence.

\begin{figure}[hbt]
	\centering
	\includegraphics[scale = 1]{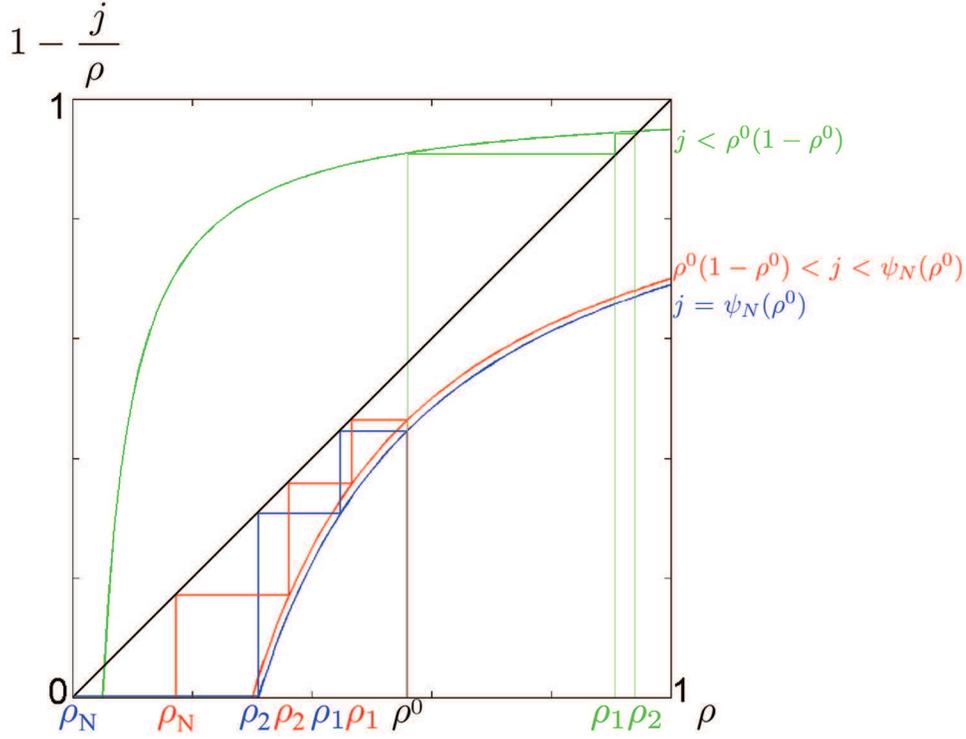}
	\caption{The cobweb plot of the sequence \eqref{J} for different values of $j$. The green, red and blue correspond to $j=0.05$, $0.3$, $0.31$, respectively; $\rho^0=0.58$.}
	\label{fig:J_rho}
\end{figure}

\begin{lemma}\label{lemmax}
	Given an $N\ge1$, there is a continuous strictly increasing function 
	$\psi_N: [0,
	1]\to \mathbb{R}_+$ satisfying
	\[
	\psi_N(\rho)>\varphi(\rho)=\rho(1-\rho), \qquad 0\le \rho\le 1,
	\]
	and the corresponding domain 
	\[ 
	\mathcal{D}_N=\left\{ (j,\rho^0): \ 
	0<j<\psi_N(\rho^0),\ \  0<\rho_0<
	1\right\},
	\]
	such that 
	\begin{itemize}
		\item The following relation holds:
		\[
		\rho_N(\psi_N(\rho^0),\rho^0)=0;
		\]
		\item For every $(j,\rho^0)\in \mathcal{D}_N$, the iterations 
		$\rho_n=\rho_n(j,\rho^0)$ defined by \eqref{J} satisfy
		\[
		\rho_0,\rho_1,\ldots,\rho_{N}>0;
		\]
		\item The iterations $\rho_0,\rho_1,\ldots,\rho_{N}$ strictly decrease 
		if $j>\rho^0(1-\rho^0)$,
		strictly increase if  $j<\rho^0(1-\rho^0)$, and satisfy 
		$\rho_0=\cdots=\rho_N=\rho^0$ if 
		$j=\rho^0(1-\rho^0)$ for $(j,\rho^0)\in \mathcal{D}_N$;
		\item
		The continuous functions $\rho_n=\rho_n(j,\rho^0): {\mathcal 
			D}_N\to \mathbb{R}$ strictly decrease in $j$ and strictly increase in 
		$\rho^0$ for all $1\le n\le N$;
		\item On the boundary of $\mathcal{D}_N$,
		\[
		\rho_N(j,\rho^0)\to 1 \quad {\rm as} \quad j\to 0+; 
		\qquad \rho_{N+1}(j,\rho^0)\to -\infty
		\quad {\rm as} \quad j\to \psi_N(\rho^0)-.
		\] 
	\end{itemize}
\end{lemma}

\begin{proof}
	The above statement directly follows from the definition of the sequence \eqref{J}, see Figure \ref{fig:J_rho}.	
\end{proof}

Two remarks are in order. 

\begin{remark}\label{rem1}
	Relations
	\[
	\rho_1=1-\frac{j}{\rho^0} >0 \quad {\rm for\ every} \quad 0 
	< j<\psi_{N}(\rho^0)
	\]
	imply
	\begin{equation}\label{psipsi}
	\psi_{N}(\rho)\le \rho, \qquad 0\le \rho\le 1.
	\end{equation}
\end{remark}

\begin{remark}\label{rem2}
	From Lemma \ref{lemmax}, it follows that iterations \eqref{J} with 
	$j=\psi_{N}(\rho^0)$ satisfy
	\begin{equation}\label{psi}
	j=\psi_{N}(\rho^0) \qquad \Rightarrow \qquad 
	\rho_0 >\rho_1>\cdots > \rho_{N-1}> 0=\rho_N,
	\end{equation}
	where each $\rho_n=\rho_n\bigl(\psi_{N}(\rho^0),\rho^0\bigr)$, 
	$n=0,\ldots,N-1$, strictly increases with $\rho^0$. 
	Moreover, each difference
	$\rho_{n-1}-\rho_{n}$ strictly increases with $\rho^0$. Indeed, 
	from the relations
	\[
	\rho_{n}-\rho_{n+1}=\frac{j}{\rho_n}-\frac{j}{\rho_{n-1}}=\frac{j(\rho_{n-1}-\rho_n)}{\rho_{n-1}{\rho_n}}=\frac{1-\rho_n}{\rho_n}(\rho_{n-1}-\rho_n)=\left(\frac{1}{\rho_n}-1\right)(\rho_{n-1}-\rho_n),
	\]
	it follows that 
	\[
	\rho_{n-1}-\rho_n=(\rho_{N-1}-\rho_N)\prod_{k=n}^{N-1}\frac{1}{\frac1{\rho_k}-1}.
	\]
	Since $\rho_N=0$ and all the $\rho_n$ increase with $\rho^0$, we see that each difference 
	\begin{equation}\label{aa}
	\rho_{n-1}-\rho_n=\rho_{n-1}\bigl(\psi_{N}(\rho^0),\rho^0\bigr)
	-\rho_n\bigl(\psi_{N}(\rho^0),\rho^0\bigr), \qquad n=1,\ldots,N,
	\end{equation}
	increases with $\rho^0$.
\end{remark}

Now, let us consider the following iterations:
\begin{eqnarray}
\label{eq:Main_Proof}
\begin{cases}
\rho_{n+1}=1-\frac{j_n}{\rho_n},\\
j_{n+1}=j_n-r_n\rho_n-\delta_n \rho_{n+1},
\end{cases}
\end{eqnarray} 
where $n=0,1,2,\ldots$\ In the following lemma, we 
compare two sequences $\{(j_n,\rho_n)\}$ and $\{(\hat j_n,\hat \rho_n)\}$, 
which start from different initial conditions and possibly have different 
sequences of coefficients $\{r_n\}$, $\{\delta_n\}$ and $\{\hat r_n\}$, $\{\hat \delta_n\}$, respectively.

\begin{lemma}\label{monotone}
	System \eqref{eq:Main_Proof} depends monotonically on the sequences 
	$\{r_n\}$, $\{\delta_n\}$ and on the initial conditions  
	in the positive quadrant in the following sense: 
	
	If
	\[
	0\le \hat r_n\leq r_n, \qquad 0\le \hat \delta_n\leq \delta_n,
	\qquad n=0,\ldots,N-1,
	\]
	\[
	\hat j_n, j_n, \hat 
	\rho_n, \rho_n>0, \qquad n=0,\ldots,N,
	\]
	and $\hat j_0\ge j_0$, $\hat 
	\rho_0\le \rho_0$, then 
	\[
	\qquad \hat j_n\ge j_n, \qquad \hat 
	\rho_n\le \rho_n, \qquad n=1,\ldots,N.
	\]
\end{lemma}

\begin{proof}
	The statement follows directly from \eqref{eq:Main_Proof} by induction in $n$.
\end{proof}

\begin{lemma}\label{new}
	Suppose 
	\begin{equation}\label{assump}
	\rho^{(2)}_0,\ldots,\rho_{k}^{(2)}>0, \qquad \rho_{k}^{(2)}>\rho_{k+1}^{(2)}, \qquad j_k^{(2)}>0.
	\end{equation}
	Then,
	\begin{equation}\label{conclu}
	\rho_0^{(1)}>\cdots>\rho_{N_1}^{(1)}>	\rho^{(2)}_0>\cdots>\rho_{k}^{(2)}>\rho_{k+1}^{(2)}.
	\end{equation}	
\end{lemma}

\begin{proof}
	By Lemma \ref{lemmax}, $R_0>R_1>\cdots>R_{N_2}=0$.
	Therefore, the sequence $j^{(2)}_n$ decreases.
	If $\rho_{n}^{(2)}\ge \rho_{n-1}^{(2)}>0$ for some $n\le k$, then due to 
	$j^{(2)}_{n-1}\ge j^{(2)}_{n}\ge j^{(2)}_k>0$,
	\[
	\rho_{n+1}^{(2)}=1-\frac{j^{(2)}_{n}}{\rho_{n}^{(2)}}\ge 
	1-\frac{j^{(2)}_{n-1}}{\rho_{n-1}^{(2)}}=\rho_{n}^{(2)}.
	\]
	That is, $\rho_{n}^{(2)}\ge \rho_{n-1}^{(2)}$ implies $\rho_{n+1}^{(2)}\ge 
	\rho_{n}^{(2)}$. Therefore \eqref{assump} implies 
	\begin{equation}\label{passing}
	\rho^{(2)}_0>\cdots>\rho_{k}^{(2)}>\rho_{k+1}^{(2)}.
	\end{equation}
	According to matching conditions \eqref{matching},
	\begin{equation}\label{AB}
	j^{(1)}=j^{(2)}_0 + J^{(2)}-\rho_0^{(2)}R_1=
	j^{(2)}_0 + R_0(1-R_1)-\rho_0^{(2)}R_1=
	j^{(2)}_0 + R_0-\rho_{N_1}^{(1)}R_1,
	\end{equation}
	and due to the relations $j_0^{(2)}=\rho_0^{(2)}(1-\rho_1^{(2)})>\rho_0^{(2)}(1-\rho_0^{(2)})$ and $R_0>R_1$,
	\[
	j^{(1)}>\rho_0^{(2)}(1-\rho_0^{(2)}) + R_0-\rho_{N_1}^{(1)}R_0,
	\]
	where $\rho_0^{(2)}=(1-c)\rho_{N_1}^{(1)}$, $R_0=c \rho_{N_1}^{(1)}$. Therefore,
	\[
	j^{(1)}>(1-c)\rho_{N_1}^{(1)} \bigl(1-(1-c)\rho_{N_1}^{(1)}\bigr)+ c \rho_{N_1}^{(1)}(1-\rho_{N_1}^{(1)}) =\rho_{N_1}^{(1)}- (c+(1-c)^2)(\rho_{N_1}^{(1)})^2,
	\]
	hence from $c<1$ it follows that $j^{(1)}>\rho_{N_1}^{(1)}(1-\rho_{N_1}^{(1)})$, which due to Lemma \ref{lemmax} implies $\rho_0^{(1)}>\cdots>\rho_{N_1}^{(1)}$. Combining this with $\rho_{0}^{(2)}=c \rho_{N_1}^{(1)}<\rho_{N_1}^{(1)}$ and \eqref{passing} proves \eqref{conclu}.
\end{proof}

\subsection{Proof of Theorem \ref{t1}}
For a given $\rho^{(1)}_0=\rho^0$, take $j^{(1)}=\psi_{N_1}(\rho^0)$. 
Then, according to Lemma \ref{lemmax},
\[
\rho^{(1)}_0 > \rho^{(1)}_1 > \cdots > \rho^{(1)}_{N_1}
=0.
\]
Lemma \ref{lemmax} further implies that given any $\varepsilon>0$, we can find a $j^{(1)}\in (\frac12\psi_{N_1}(\rho^0),\psi_{N_1}(\rho^0))$, such that sequence \eqref{problem} satisfies
\begin{equation}\label{ro}
\rho^0=\rho^{(1)}_0 > \rho^{(1)}_1 > \cdots > \rho^{(1)}_{N_1}
> 0; \qquad \rho^{(1)}_{N_1}<\varepsilon.
\end{equation}
Therefore, matching conditions \eqref{matching} imply
\[
\rho_0^{(2)}=(1-c)\rho_{N_1}^{(1)}<\varepsilon, \qquad J^{(2)}<R_0=c\rho_{N_1}^{(1)}<\varepsilon
\]
and
\[
j_0^{(2)}=j^{(1)}-J^{(2)}+\rho_0^{(2)}R_1>\frac12\psi_{N_1}(\rho^0)-\varepsilon,
\]
and from \eqref{problem} it follows that
\[
\rho_1^{(2)}<1-\frac{j_0^{(2)}}\varepsilon.
\]
Therefore, for a sufficiently small $\varepsilon$, we obtain 
\begin{equation}\label{basis}
\rho^{(2)}_{0}>0 > \rho^{(2)}_{1}, \qquad j^{(2)}_0>0.
\end{equation}

We now show by induction that if for some $k\ge 0$ and $A_k>0$,
\begin{equation}\label{induction}
j^{(1)}=A_k \quad \Rightarrow \quad 
\rho^{(1)}_0 > \cdots > \rho^{(1)}_{N_1} >
\rho^{(2)}_0 > \cdots > \rho^{(2)}_k > 0 > \rho^{(2)}_{k+1}; \quad j^{(2)}_k>0,
\end{equation}
then there is an $A_{k+1}\in (0,A_k)$ such that
\begin{equation}\label{induction'}
j^{(1)}=A_{k+1} \quad \Rightarrow \quad 
\rho^{(1)}_0 >  \cdots > \rho^{(1)}_{N_1}>\rho^{(2)}_0> \cdots > \rho^{(2)}_{k+1} > 0 > \rho^{(2)}_{k+2}; \quad 
j^{(2)}_{k+1}>0.
\end{equation}

Relations \eqref{ro}, \eqref{basis} establish the basis for the induction. 

For the induction step, assume that \eqref{induction} holds for a $k\ge 0$.
Let us consider how the values $\rho^{(1)}_n=\rho^{(1)}_n(j^{(1)})$ and 
$\rho_n^{(2)}=\rho^{(2)}_n(j^{(1)})$ change when $j^{(1)}$ decreases from 
the value $j^{(1)}=A_k$. By Lemma \ref{lemmax}, $\rho^{(1)}_{N_1}$ 
increases with decreasing $j^{(1)}$.
Hence, $\rho_0^{(2)}=(1-c)\rho_{N_1}^{(1)}$, $R_0=c\rho^{(1)}_{N_1} $ increase, and according to Remark 
\ref{rem2}, 
all $R_n=R_n\bigl(\psi_{N_2}(R_0),R_0\bigr)$ and all the differences $R_n-R_{n+2}$ with $n=0,\ldots,N_2-1$  also 
increase. Further, \eqref{AB} implies that 
\[
j^{(1)}=j^{(2)}_0 +(R_0-R_1)\rho^{(1)}_{N_1}+(1-\rho^{(1)}_{N_1})R_0=
j^{(2)}_0 +(R_0-R_1)\rho^{(1)}_{N_1}+c(1-\rho^{(1)}_{N_1})\rho^{(1)}_{N_1},
\]
where $R_0-R_1$, $\rho^{(1)}_{N_1}$ and $(1-\rho^{(1)}_{N_1})\rho^{(1)}_{N_1}$ increase with decreasing $j^{(1)}$ as long as $\rho^{(1)}_{N_1}<1/2$, and hence $j_0^{2}$ decreases.
Therefore, Lemma \ref{monotone}
implies that all $\rho^{(2)}_n$, $n=0,\ldots,k$
increase, and in particular 
$\rho^{(2)}_0,\ldots, \rho^{(2)}_k$ remain positive, while all $j_n^{(2)}$, 
$n=0,\ldots,k$  decrease,  as long as $j_k^{(2)}$ remains positive  with 
decreasing $j^{(1)}$.
Combining this with Lemma \ref{new}, we see that the ordering
\[
1/2\ge 
\rho^{(1)}_0 > \cdots > \rho^{(1)}_{N_1} >
\rho^{(2)}_0 > \cdots > \rho^{(2)}_k >0
\]
is preserved, and all $\rho^{(1)}_0,\ldots,\rho^{(1)}_{N_1}, \rho^{(2)}_0, \ldots, \rho^{(2)}_k  $ increase with decreasing $j^{(1)}$ as long as
the relations $\rho^{(2)}_k >\rho^{(2)}_{k+1}$, $j^{(2)}_k>0$ remain valid.
Now, note that $j^{(1)}>j^{(2)}_0\ge \cdots \ge j^{(2)}_k$ due to the second equation of \eqref{problem''} and \eqref{AB}.
Hence, using the continuous dependence of $j^{(2)}_k$ on $j^{(1)}$, we 
conclude that 
there is a value $\tilde A\in[0, A_k)$ such that  $j^{(2)}_k\to 0+$ as 
$j^{(1)}\to \tilde A+$.
Further,
\[
\rho^{(2)}_{k+1}=1-\frac{j_k^{(2)}}{\rho_k^{(2)}}\to 1 \quad {\rm as} \quad 
j^{(1)}\to \tilde A+
\]
because $\rho_k^{(2)}(\tilde A)>\rho_k^{(2)}(A_k)>0$.
Combining this with $\rho^{(2)}_{k+1}(A_k)<0$, which is one of the 
induction assumptions,
and using the continuity of $\rho^{(2)}_{k+1}(j^{(1)})$,
we see that there is an $\hat A\in (\tilde A,A_k)\subset(0,A_k)$
such that $\rho^{(2)}_{k+1}(\hat A)=0$ or, equivalently, $j^{(2)}_k(\hat 
A)=\rho^{(2)}_k(\hat A)$. Hence,
\begin{equation}\label{ut}
j^{(1)}=\hat A \quad \Rightarrow \quad 
\rho^{(2)}_0 > \cdots > \rho^{(2)}_k > 0 = \rho^{(2)}_{k+1}; \quad j_k^{(2)}=\rho_k^{(2)}>0; \quad 
j_{k+1}^{(2)}>0,
\end{equation}
where the last inequality follows from the relations 
$j_{k+1}^{(2)}=j_k^{(2)}-R_{k+1} \rho_k^{(2)}=(1-R_{k+1})\rho_k^{(2)}$ 
in which
$R_{k+1}\le R_0= c\rho^{(1)}_{N_1} \le c<1$.
Again, using the continuity and monotonicity argument, relations \eqref{ut}
imply
\begin{equation}\label{ut'}
\rho^{(2)}_{k+1}\to 0+,\qquad 
\rho^{(2)}_{k+2}=1-\frac{j^{(2)}_{k+1}}{\rho^{(2)}_{k+1}}\to -\infty \quad \
{\rm as} \quad \  j^{(1)}\to \hat A-.
\end{equation}
Thus, relations \eqref{ut}, \eqref{ut'} and Lemma \ref{new} imply \eqref{induction'} for 
$A_{k+1}=\hat A-\delta>0$ with a sufficiently small $\delta>0$, which 
completes the proof of the induction step.

Finally, the induction ensures the existence of a $j^{(1)}>0$ for which 
$j^{(2)}_{N_2+N_3-1}>0$, all $\rho^{(2)}_n$, $n=0,\ldots, N_2+N_3-1$ are 
positive, and $\rho^{(2)}_{N_2+N_3}<0$.
Repeating the argument that was already used above in the proof of the 
induction step, one can now decrease $j^{(1)}$
to a smaller positive value to achieve 
$\rho^{(1)}_0 > \cdots > \rho^{(1)}_{N_1} > \rho^{(2)}_0 > \cdots > 
\rho^{(2)}_{N_2+N_3} =0$. Relations $ R_0 > \cdots > R_{N_2}=0$
follow from Lemma \ref{lemmax}.
This completes the proof of Theorem \ref{t1}.
{}\hfill $\Box$

\subsection{Proof of Theorem \ref{t2}}
Let \eqref{solution} be a positive decreasing solution of problem
\eqref{problem}\,--\,\eqref{matching} with $\rho^0\in (0, 1/2)$.
First, let us show that
\begin{equation}\label{otsen}
j_m^{(2)}-\varphi(\rho_m^{(2)})= O(N_2^{-1}), \qquad 
0\le m\le N_2 + \frac{\ln N_2}{\ln q}.
\end{equation}
Recall that $x_m=O(N_2^\alpha)$ is equivalent
to $|x_m|\le C N_2^\alpha$ where $C$ is independent of $N_2$ for the given range $m$.

Using function \eqref{varphi}, formula \eqref{problem'} can be rewritten as
\begin{equation}\label{pr'}
\rho_{m+1}^{(2)}=\rho_m^{(2)}+\frac{\varphi(\rho_m^{(2)})-j_m^{(2)}}{\rho_m^{(2)}}.
\end{equation}
Introducing the notation
\[
y_m=j_m^{(2)}-\varphi(\rho_m^{(2)}),
\]
we see that 
\[
y_{m+1}-y_m=j_{m+1}^{(2)}-j_m^{(2)} + \varphi(\rho_m^{(2)})-\varphi\left(\rho_m^{(2)}-\frac{y_m}{\rho_m^{(2)}}\right)
\]
and using an intermediate value of the derivative of $\varphi$,
\[
y_{m+1}-y_m=j_{m+1}^{(2)}-j_m^{(2)} + \frac{y_m}{\rho_m^{(2)}}\, \varphi'\left(\rho_m^{(2)}-\frac{\theta_m y_m}{\rho_m^{(2)}}\right),
\]
where $\theta_m\in[0,1]$. Due to concavity of the function $\varphi$ and relations $\rho_m^{(2)}<\rho^0<1/2$, we have 
\[
\frac{1}{\rho_m^{(2)}}\, \varphi'\left(\rho_m^{(2)}-\frac{\theta_m y_m}{\rho_m^{(2)}}\right)\ge \frac{\varphi'(\rho_m^{(2)})}{\rho_m^{(2)}}=\frac{1-2\rho_m^{(2)}}{\rho_m^{(2)}}> \frac1{\rho^0}-2>0, 
\]
hence
\[
y_{m+1}\ge j_{m+1}^{(2)}-j_m^{(2)} + \frac{ y_m}q, \qquad q:=\frac{\rho^0}{1-\rho^0}<1.
\]
Consequently,
\[
y_m \le q^{N_2-m} y_{N_2}+\frac{q}{1-q} \max_{0\le m\le N_2} (j_m^{(2)}-j_{m+1}^{(2)}),
\]
where $y_{N_2}=j_{N_2}^{(2)}-\varphi(\rho_{N_2}^{(2)})\le j_{N_2}^{(2)}\le \rho_{N_2}^{(2)}<\rho^0$ and
according to \eqref{problem''},
\[
j_m^{(2)}-j_{m+1}^{(2)}=R_{m+1}\rho_m^{(2)}+(R_{m}-R_{m+2})\rho^{(2)}_{m+1}
\le 2R_0\rho^{(2)}_0=2 c(1-c)(\rho^{(1)}_{N_1})^2\le \frac{2c_0 (\rho^0)^2}{N_2}.
\]
Therefore,
\[
y_m \le q^{N_2-m} \rho^0+\frac{2 q c_0(\rho^0)^2}{(1-q)N_2}.
\]
In particular,
\[
y_m \le \frac{\rho^0}{N_2} \left(1+\frac{2 q c_0 \rho^0}{1-q}\right),  \qquad
0\le m\le N_2 + \frac{\ln N_2}{\ln q},
\]
where 
$y_m=j_m^{(2)}-\varphi(\rho_m^{(2)})=\rho_m^{(2)}(1-\rho_{m+1}^{(2)})-
\rho_m^{(2)}(1-\rho_m^{(2)})\ge 0$. This proves \eqref{otsen}.

Next, we obtain estimates for the differences $R_m-R_{m+1}>0$ and $R_0-R_m>0$.
Relation \eqref{aa} implies that
\[
R_{m}-R_{m+1}=\frac{R_{N_2-1}-R_{N_2}}{\left(\frac1{R_0}-1\right)^{N_2-m-1}}\le {R_0}{\left(\frac{R_0}{1-R_0}\right)^{N_2-m-1}}.
\]
Therefore,
\[
R_{0}-R_{m}\le \frac{R_0}{1-\frac{R_0}{1-R_0}}{\left(\frac{R_0}{1-R_0}\right)^{N_2-m}}.
\]
Since $R_0=c \rho^{(1)}_{N_1}\le c \rho^0=c_0\rho^0/N_2$, it follows that
\begin{equation}\label{RR}
R_{m}-R_{m+1}=O(N_2^{m-N_2}), \qquad R_{0}-R_{m}=O(N_2^{m-N_2-1}).
\end{equation}

Now, let us consider the initial value problem
\begin{equation}\label{dif}
\frac{d \hat j}{\tau}=-c_0\rho^0 \hat\rho(\tau), \qquad \hat j =\varphi(\hat \rho)=\hat \rho(1-\hat\rho); \qquad\quad \hat j(0)=\varphi(\rho_0^{(2)})
\end{equation}
on the interval $0\le \tau\le 1$.
Its solution is given by
\begin{equation}\label{WW}
\hat j(\tau)=\varphi(\hat\rho(\tau)), \qquad \hat\rho(\tau)=-\frac12\, W\left(-2\rho_0^{(2)}e^{-2\rho_0^{(2)}-c_0\rho^0\tau}\right).
\end{equation}
Denote by $\varphi^{-1}$ the inverse of the restriction $\varphi: [0,1/2]\to \mathbb{R}_+$ of function \eqref{varphi} to its interval of monotonicity $[0,1/2]$. Set
\[
\hat j_m =\hat j(m h), \qquad \hat \rho_m =\hat \rho(mh)=\varphi^{-1}(\hat j_m), \qquad m=0,1,\ldots, N_2; \qquad h:=1/N_2.
\]
Since $\hat\rho(\tau)$ decreases, \eqref{dif} implies
\begin{equation}\label{ivt}
\hat j_{m+1}-\hat j_m=-h c_0\rho^0  (\chi_m \hat \rho_m +(1-\chi_m)\hat\rho_{m+1}) ,
\end{equation}
where $\chi_m\in[0,1]$ for each $m=0,\ldots,N_2-1$.
Let us
estimate the error
\[
e_m=j^{(2)}_m-\hat j_m.
\]
Subtracting \eqref{ivt} from \eqref{problem''} gives
\[
\begin{array}{rcl}
e_{m+1}-e_m &=& -R_{m+1}\rho_m^{(2)}-(R_{m}-R_{m+2})\rho^{(2)}_{m+1} +h c_0\rho^0  (\chi_m \hat \rho^m +(1-\chi_m)\hat\rho_{m+1}) \\
&=& (R_0-R_{m+1})\rho_m^{(2)}-(R_{m}-R_{m+2})\rho^{(2)}_{m+1} +h c_0\rho^0  (1-\chi_m)  (\hat\rho_{m+1}-\hat \rho_m)\\ &  + & (h c_0\rho^0-R_0 )\rho_m^{(2)} + h c_0\rho^0   (\hat \rho_m -\rho_m^{(2)}).
\end{array}
\]
Combining this with $h=1/N_2$ and \eqref{RR}, we obtain
\[
e_{m+1}-e_m=
O(N_2^{m+1-N_2})-h c_0\rho^0  (1-\chi_m)  (\hat \rho_m-\hat\rho_{m+1}) +  (h c_0\rho^0-R_0 )\rho_m^{(2)} + h c_0\rho^0   (\hat \rho_m -\rho_m^{(2)}).
\]
From
\[
0\le \hat \rho_m-\hat\rho_{m+1}=-h\frac{d\hat\rho}{d\tau}(\tau_m)=\frac{-h}{1-2\hat\rho(\tau_m)}\cdot\frac{d\hat j}{d\tau}(\tau_m)=h^2 c_0 \rho^0\frac{\hat\rho(\tau_m)}{1-2\hat\rho(\tau_m)},
\]
where $mh\le \tau_m\le (m+1)h$ and $\hat\rho(\tau_m)\le \hat\rho_m\le\hat\rho_0=\rho_0^{(2)}\le \rho^0$, it follows that
\[
0\le \hat \rho_m-\hat\rho_{m+1}\le \frac{h^2 c_0 (\rho^0)^2}{1-2\rho^0},
\]
hence
\begin{equation}\label{aaa}
e_{m+1}-e_m=
O(N_2^{m+1-N_2})+O(N_2^{-2}) +  (h c_0\rho^0-R_0 )\rho_m^{(2)} + h c_0\rho^0   (\hat \rho_m -\rho_m^{(2)}).
\end{equation}
According to matching conditions \eqref{matching},
\[
j^{(1)}=j^{(2)}_0 -\rho_0^{(2)}R_1+J^{(2)}=j^{(2)}_0 -\rho_0^{(2)}R_1+R_0(1-R_1),
\]
which together with \eqref{otsen} implies 
\begin{equation}\label{bb'}
j^{(1)}=j^{(2)}_0+O(N_2^{-1}), 
\end{equation}
hence
\[
j^{(1)}=\rho_0^{(1)}(1-\rho_1^{(1)})\ge \rho_0^{(1)}(1-\rho_0^{(1)})=\varphi(\rho^0)\ge \varphi(\rho_0^{(2)})=j^{(1)}+ O(N_2^{-1})
\]
and therefore
\begin{equation}\label{bb}
\varphi(\rho^0)=j^{(1)}+O(N_2^{-1}), \qquad \varphi(\rho_0^{(2)})=j^{(1)}+O(N_2^{-1}).
\end{equation}
Due to $0<j^{(1)}<\rho^0<1/2$, from these relations and the relations
\[
c_0\rho^0-N_2 R_0=c_0(\rho^0-\rho_{N_1}^{(1)})=c_0\left(\rho^0-\frac{\rho_0^{(2)}}{1-\frac{c_0}{N_2}}\right)=c_0\left(\rho^0-{\rho_0^{(2)}}\right)+O(N_2^{-1}),
\]
it follows that
$
c_0\rho^0-N_2 R_0=O(N_2^{-1}),
$
and \eqref{aaa} implies
\[
e_{m+1}-e_m=
\frac{c_0\rho^0   (\hat \rho_m -\rho_m^{(2)})}{N_2}+O(N_2^{-2})
\]
for $m\le N_2-3$. Combining this with $\hat\rho_m=\varphi^{-1}(\hat j_m)$ and \eqref{otsen} gives
\[
e_{m+1}-e_m=
c_0\rho^0\,\frac{   \varphi^{-1}(\hat j_m)  - \varphi^{-1}(j_m^{(2)}+O(N_2^{-1}))}{N_2}+O(N_2^{-2})
\]
for $0\le m\le N_2 + \frac{\ln N_2}{\ln q}$, and due to
\[
\frac{d\varphi^{-1}}{dj}(j)=\frac1{1-2\varphi^{-1}(j)}, \qquad 0<\hat\rho_m, \rho_m^{(2)}\le \rho_0^{(2)}<\rho^0,
\]
we obtain
\[
|e_{m+1}-e_m|\le 
\frac{c_0\rho^0|\hat j_m-j_m^{(2)}|}{(1-2\rho^0)N_2}+O(N_2^{-2})=
\frac{c_0\rho^0|e_m|}{(1-2\rho^0)N_2}+O(N_2^{-2}).
\]
By discrete Gr\"onwall's inequality, this implies
\[
|e_m|\le e^\frac{\kappa m}{N_2}|e_0|+\frac{N_2 (e^\frac{\kappa m}{N_2}-1)}{\kappa}\cdot O(N_2^{-2})=e^\frac{\kappa m}{N_2} |e_0|+O(N_2^{-1}), \qquad \kappa:=\frac{c_0\rho^0}{1-2\rho^0},
\]
with $e_0=j_0^{(2)}-\hat j_0=j_0^{(2)}-\varphi(\rho_0^{(2)})=O(N_2^{-1})$, where the last equality follows from \eqref{otsen}, hence
\begin{equation}\label{e1}
e_m=O(N_2^{-1}) \quad \text{for all} \quad 0\le m\le N_2 + \frac{\ln N_2}{\ln q}.
\end{equation}
Finally, relations \eqref{problem''} and \eqref{dif} ensure that
$j_{m}^{(2)}-j_{m+1}^{(2)}=O(N_2^{-1})$ and $\hat j_{m}-\hat j_{m+1}=O(N_2^{-1})$, hence
\begin{equation}\label{e2}
e_m=O(N_2^{-1}\ln N_2) \quad \text{for all} \quad 0\le m\le N_2  
\end{equation}
Combining \eqref{e1}, \eqref{e2} with relations \eqref{otsen}, \eqref{WW}, \eqref{bb'}, \eqref{bb} and $\rho^0=\rho_0^{(1)}>\cdots>\rho_{N_1}^{(1)}>\rho_0^{(2)}$, we obtain \eqref{lim}\,--\,\eqref{fj3sq}. This completes the proof.

\section{Conclusion}
We proposed a deterministic model for the translation of an uORF with leaky scanning initiation. The model assumes that the collision of elongating and scanning ribosomes can result in detachment of the latter from the mRNA. We showed that the dependence of the flow of scanning ribosomes downstream of the uORF on the density of scanning ribosomes upstream of the uORF for the stationary solution is non-monotone provided that the uORF is sufficiently long and the fraction of the initiated ribosomes is sufficiently small. In particular, reducing the ribosome flow upstream of the uORF can increase the flow downstream. Further, we obtained an explicit formula for the flow downstream of the uORF in the limit when the length $N$ of the uORF increases and the initiation efficiency decreases as $1/N$ (see \eqref{varphi}, \eqref{W'} and \eqref{fj3sq}). These analytical results agree with earlier numerical findings obtained by TASEP modeling under similar assumptions \citep{40}. We showed numerically that the downstream flow predicted by the deterministic model and the TASEP model can be relatively close to our analytical solution for realistic values of $N$. Therefore, our model provides a plausible explanation for the experimentally confirmed ability of a single uORF to upregulate the downstream ORF translation when the upstream flow of ribosomes is reduced due to stress. We associate the mechanism of upregulation with the above stated assumption that some scanning ribosomes dissociate from the uORF before reaching the stop codon as a result of interaction with the elongating ribosomes. Our results predict that stress resistance can be achieved by translation of relatively long uORFs that do not favor high levels of translation initiation. In general these findings are relevant not only to IRS, but to any changes in intracellular conditions that affect the flow of scanning ribosomes or parameters of uORF translation, e.g., induction of a ribosome pause at uORF or uORF initiation rate. \vskip 6pt

\section*{Acknowledgments}
The authors thank M. Arnold for useful discussion of the models and results. This work was supported by Russian Science Foundation (20-14-00121).

\end{document}